\documentclass[11pt,a4paper,leqno]{article}
\usepackage[DIV=11]{typearea}
\usepackage[american]{babel}
\usepackage[utf8]{inputenc}
\usepackage{hyperref}
\usepackage{algpseudocode}
\usepackage{amsmath}
\usepackage{amssymb}
\usepackage{amsfonts}
\usepackage{amsthm}
\usepackage{amscd}
\usepackage{mathdots}
\usepackage{tikz}
\usepackage{dynkin-diagrams}
\usepackage[shortlabels]{enumitem}
\usepackage{pgfplots}
\usetikzlibrary{positioning}
\pgfplotsset{compat=1.8}
\usepackage{subfig}
\usepackage{bbold}
\usepackage{fancyhdr}
\usepackage[title]{appendix}

\usepackage[noblocks]{authblk}
\usepackage{verbatim}
\makeatletter
\def\revddots{\mathinner{\mkern1mu\raise\p@
\vbox{\kern7\p@\hbox{.}}\mkern2mu
\raise4\p@\hbox{.}\mkern1mu\raise7\p@\hbox{.}\mkern1mu}}
\makeatother

\theoremstyle{plain}
\newtheorem{thm}{Theorem}[section]
\newtheorem{lem}[thm]{Lemma}
\newtheorem{prop}[thm]{Proposition}
\newtheorem{cor}[thm]{Corollary}

\theoremstyle{definition}
\newtheorem*{defn}{Definition}

\newtheorem{rem}[thm]{Remark}

\theoremstyle{remark}

\newcommand{\cind}{{\textnormal{c-ind}}}
\newcommand{\Hom}{{\textnormal{Hom}}}
\newcommand{\End}{{\textnormal{End}}}
\newcommand{\Res}{{\textnormal{Res}}}

\newcommand{\F}{F}
\newcommand{\GL}{\textnormal{GL}}

\newcommand{\A}{\mathcal{A}}
\renewcommand{\H}{\mathcal{H}}

\newcommand{\SO}{\textnormal{SO}}

\title{The Gelfand--Graev representation of $\SO(2n+1)$\\
in terms of Hecke algebras.}
\author{Petar Baki\'c, Gordan Savin}
\date{}
\begin{document}

\maketitle

\begin{abstract}
Let $G$ be a $p$-adic classical group. The representations in a given Bernstein component can be viewed as modules for the corresponding Hecke algebra---the endomorphism algebra of a pro-generator of the given component. Using Heiermann's construction of these algebras, we describe the Bernstein components of the Gelfand--Graev representation for $G=\SO(2n+1)$. 
\end{abstract}

\section{Introduction}
 Let $\F$ be a non-Archimedean local field of residue characteristic $q$. Let $G$ be the group of $F$-points of a connected, split reductive algebraic group defined over $\F$; in particular, the group $G$ contains a Borel subgroup. Let 
 $U$ be the unipotent radical of the Borel subgroup and fix a non-degenerate (Whittaker) character $\psi:  U \rightarrow \mathbb C^{\times}$. The Gelfand--Graev representation of $G$ is $\cind_U^G (\psi)$ where $\cind$ stands for induction with compact support.  The goal of this paper is to give an explicit description of the Bernstein components of the Gelfand--Graev representation. 
 
 \vskip 5pt 
 Let us briefly describe what is known. Let $K$ be a special maximal compact subgroup of $G$ and  $I$ an Iwahori subgroup contained in $K$. 
Let $\mathcal H$ be the Iwahori--Hecke algebra of $I$-biinvariant functions on $G$, and $\mathcal H_K$ the subalgebra consisting of functions supported on $K$. 
Then $\mathcal H_K$ is isomorphic to the group algebra of the Weyl group $W$ of $G$ and thus it has a one-dimensional representation $\epsilon$ (the sign character). 
As an $\mathcal H$-module, $(\cind^{G}_U \psi)^I$  is isomorphic to the projective $\H$-module \cite{chan2018iwahori} 
\[ 
\mathcal H\otimes_{\mathcal H_K} \epsilon.
\]
If $G=\GL_n$  then a similar statement holds for all indecomposable Bernstein components  with appropriate Hecke algebras arising from Bushnell--Kutzko types 
 \cite{chan2019bernstein}. We build on methods of that paper. We  finish this paragraph by mentioning a recent article of Mishra and Pattanayak \cite{mishra2020principal} that  
 considers Bernstein components of $\cind_U^G (\psi)$ corresponding to representations induced from the Borel subgroup. Their result is formulated in terms 
 of Hecke algebras arising form types constructed by Roche.   
 
 \vskip 5pt 
  
 For a general $G$ one does not have a complete theory of types and corresponding Hecke algebras, but there is a replacement: endomorphism algebras 
of pro-generators of Bernstein components.  

It turns out that these algebras are more suited for the problem at hand. 
In more details, let $P=MN$ be a parabolic subgroup of $G$, and let $\sigma$ be an irreducible cuspidal representation of $M$. Let $M^\circ$  be the subgroup of $M$ consisting of all $m\in M$ such that 
$|\chi(m)|=1$ for all smooth characters $\chi: M \rightarrow \mathbb C^{\times}$.  Let $\sigma_0$ be an irreducible summand of $\sigma $ restricted to 
$M^\circ$. Then $i_P^G( \cind^M_{M^\circ} (\sigma_0) )$ is a projective $G$-module generating a single Bernstein component. 
Here $i_P^G$ denotes normalized parabolic induction. Let 
\[ 
\mathcal H = {\rm End}_G( i_P^G( \cind^M_{M^\circ} (\sigma_0) )). 
\] 
Observe that  we have  a natural inclusion 
\[ 
\mathcal A  = {\rm End}_M( \cind^M_{M_0} (\sigma_0) ) \subseteq \mathcal H. 
\] 
For every $G$-module $\pi$, 
\[ 
\mathfrak F(\pi) = \Hom_G ( i_P^G( \cind^M_{M^\circ} (\sigma_0) ), \pi) 
\] 
is naturally a right $\mathcal H$ module. The functor $\mathfrak F$ is an equivalence between the Bernstein component generated by  
$i_P^G( \cind^M_{M^\circ} (\sigma_0) )$ and the category of right $\mathcal H$-modules. 

Now assume $\sigma$ is $\psi$-generic. Let 
\[ 
\Pi= \mathfrak F( \cind_U^G (\psi)). 
\] 
It is not difficult to see, using Bernstein's second adjointness, that $\Pi \cong \mathcal A $, as $\mathcal A$-modules. Thus understanding $\Pi$ reduces to 
understanding $\mathcal H$ modules isomorphic to $\mathcal A$. This was done for $\GL_n$ in \cite{chan2019bernstein}. We extend this 
computation to $\mathcal H$ for classical groups.  For classical groups the algebra $\mathcal H$ has been computed by Heierman \cite{heiermann2011operateurs}. 
It turns out that a general $\mathcal H$ is a tensor product of Hecke algebras each of which is isomorphic to the Iwahori-Hecke algebra $\GL_n$ or to an algebra of the type 
$\tilde C_n$, with unequal parameters. Assume that $\mathcal H$ corresponds to $\tilde C_n$. Its diagram has two special vertices, denoted by $0$ and $n$. Corresponding 
to them, we have two finite subalgebras $\mathcal H_0$ and $\mathcal H_n$ of $\mathcal H$. We prove that any $\mathcal H$-module isomorphic to $\mathcal A$ is necessarily 
\[ 
\mathcal H \otimes_{\mathcal H_0} \epsilon_0 \text{ or } \mathcal H \otimes_{\mathcal H_n} \epsilon_n
\] 
for a one-dimensional representation $\epsilon_0$ or $\epsilon_n$.  Here we moved to more familiar language of left $\mathcal H$-modules. This is harmless indeed, since
$\mathcal H$ is isomorphic to its opposite algebra; this follows from the Iwahori-Matsumoto relations. If $G=\SO(2n+1)$  we determine precisely the isomorphism 
class of $\Pi$. 

\vskip 5pt 
We finish this paper with two appendices, added as a convenience to the reader. 
The first one gives a proof of the fact that the functor $\mathfrak F$ is the claimed equivalence. The second gives an isomorphism of $\mathcal H$ with the 
Hecke algebra arising from the type constructed by Stevens. 

\vskip 5pt 
We would like to thank K.-Y. Chan and S. Stevens for useful communications. 
G. Savin is partially supported by a grant from the National Science Foundation, DMS-1901745.

\section{Preliminaries}
\subsection{Notation}
Throughout the paper, $F$ will denote a non-Archimedean local field of residue characteristic $q$ and uniformizer $\varpi$, equipped with the absolute value $|\cdot|$ normalized in the usual way.

We let $G$ denote the group of $F$-points of a connected, split reductive algebraic group defined over $\F$. From \S 2.4 on, we specialize to the case where $G$ is the special odd orthogonal group $\SO(2N+1)$. By $\text{Rep}(G)$ we denote the category of smooth complex representations of $G$.

For an arbitrary group $H$, we let $X(H)$ denote the group of complex characters of $H$.

\subsection{Unramified characters}
\label{subs_unramified}
If $M$ is a Levi subgroup of $G$, we let $M^\circ = \bigcap_{\chi} \ker |\chi|$, the intersection taken over the set of all rational characters $\chi: M \to F^\times$. We say that a (complex) character $\chi$ of $M$ is unramified if it is trivial on $M^\circ$; we let $X^\text{nr}(M)$ denote the group of all unramified characters on $M$. Then $M/M^\circ$ is a free $\mathbb{Z}$-module of finite rank, and the group  $X^\text{nr}(M) = X(M/M^\circ)$ has a natural structure of a complex affine variety. For any element $m\in M$, we denote by $b_m$ the evaluation $\chi \mapsto \chi(m)$.

Now let $\sigma$ be an irreducible cuspidal representation of $M$, and set $M^\sigma = \{m \in M : {}^m{\sigma} \cong \sigma\}$. Then $M/M^\sigma$ is a finite Abelian group, and we let $\A$ denote the ring of regular functions on the quotient variety $X(M/M^\circ)/X(M/M^\sigma)$. Since $M^\sigma/M^\circ$ is once again a free $\mathbb{Z}$-module (of the same rank as $M/M^\circ$), we have $\A \cong \mathbb{C}[M^\sigma/M^\circ]$. Furthermore, letting $\sigma_0$ denote an arbitrary irreducible constituent of $\sigma|_{M^\circ}$, we have a canonical isomorphism $\A \cong \End_M(\cind_{M^\circ}^M \sigma_0)$. Indeed, this follows from a simple application of Mackey theory. We refer the reader to \cite[\S 1.17, \S 4]{heiermann2011operateurs} for additional details.

\subsection{The Hecke algebra of a Bernstein component}
\label{subs_Hecke}
If $\pi$ is an irreducible representation of $G$, there is a Levi subgroup $M$ of $G$ and an irreducible cuspidal representation $\sigma$ of $M$ such that $\pi$ is (isomorphic to) a subquotient of $i_P^G(\sigma)$. Here $P$ is a parabolic subgroup of $G$ with Levi component $M$. The pair $(M,\sigma)$ is determined by $\pi$ up to conjugacy; we call $(M,\sigma)$ the cuspidal support of $\pi$.

We say that the two pairs $(M_1,\sigma_1)$ and $(M_2,\sigma_2)$ as above are inertially equivalent if there exists an element $g\in G$ and an unramified character $\chi$ of $M_2$ such that
\[
{}^g\sigma_1 = \sigma_2 \otimes \chi.
\]
This is an equivalence relation on the set of all pairs $(M,\sigma)$.
Given an equivalence class $[(M,\sigma)]$, we denote by $\text{Rep}_{(M,\sigma)}(G)$ the full subcategory of $\text{Rep}(G)$ defined by the requirement that all irreducible subquotients of every object in $\text{Rep}_{(M,\sigma)}(G)$ be supported within the inertial class $[(M,\sigma)]$.
A classic result of Bernstein then shows that the category $\text{Rep}(G)$ decomposes as a direct product
\[
\text{Rep}(G) = \prod_{[(M,\sigma)]} \text{Rep}_{(M,\sigma)}(G)
\]
taken over the set of all inertial equivalence classes. We refer to $\text{Rep}_{(M,\sigma)}(G)$ as the Bernstein component attached to the pair $(M,\sigma)$. For a detailed discussion of the above results, see \cite{bernstein1984centre} or \cite{bernsteindraft}.

For each Bernstein component $\text{Rep}_{(M,\sigma)}(G)$ one can construct a projective generator $\Gamma_{(M,\sigma)}$ by setting
\[
\Gamma_{(M,\sigma)} = i_P^G(\cind_{M^\circ}^M(\sigma_0)).
\]
Here $\sigma_0$ is any irreducible component of the (semisimple) restriction $\sigma|_{M^\circ}$. We now obtain a functor from the category $\text{Rep}_{(M,\sigma)}(G)$ to the category of right $\End_G(\Gamma_{(M,\sigma)})$-modules given by
\[
\pi \mapsto \Hom(\Gamma_{(M,\sigma)},\pi).
\]
The fact that $\Gamma_{(M,\sigma)}$ is a projective generator implies that this is an equivalence of categories. A detailed explanation of this fact is provided in Appendix \ref{sec_appA}, Theorem \ref{thm_A4}.

Given a Bernstein component $\mathfrak{s} = (M,\sigma)$, we use $\H_\mathfrak{s}$ to denote $\End_G(\Gamma_\mathfrak{s})$ and refer to it as the Hecke algebra attached to the component $\mathfrak{s}$. Furthermore, for any $\pi \in \text{Rep}(G)$ we let $\pi_\mathfrak{s}$ denote the corresponding $\H_\mathfrak{s}$-module $\Hom(\Gamma_{\mathfrak{s}},\pi)$.

Though we do not use it here, we point out that there is another highly useful approach to analyzing Bernstein components, based on the theory of types developed by Bushnell and Kutzko \cite{bushnell1998smooth}. One can show that the Hecke algebra used by Bushnell and Kutzko is in fact isomorphic to the algebra $\H_\mathfrak{s}$ introduced above; we prove this fact in Appendix \ref{sec_appB}. Therefore---for the purposes of this paper---the two approaches are equivalent.

\subsection{Cuspidal representations}
\label{subs_cusp}
Here we briefly recall some facts and introduce notation related to cuspidal representations. For the sake of concreteness we now specialize to $G=\SO(2N+1)$, but we point out that these facts generalize to all classical $p$-adic groups.

Let $\rho$ and $\sigma$ be irreducible unitarizable cuspidal representations of $\GL_k(F)$ and $\SO(2n_0+1)$ respectively. We consider the representation $\nu^\alpha\rho \rtimes \sigma$, where $\alpha \in \mathbb{R}$. Here and throughout the paper, we use $\nu$ to denote the unramified character $\lvert\det\rvert$ of the general linear group. If $\rho$ is not self-dual, the above representation never reduces. If $\rho$ is self-dual, then there exists a unique $\alpha \geq 0$ such that $\nu^\alpha\rho \rtimes \sigma$ is reducible; we denote it by $\alpha_\rho$. The number $\alpha_\rho$ has a natural description in terms of the Langlands parameters: if $a_\rho = \max\{a: \rho \otimes S_a \text{ appears in the parameter of }\sigma\}$, where $S_a$ denotes the (unique) irreducible algebraic $a$-dimensional representation of $SL_2(F)$, then $\alpha_\rho = \frac{a_\rho+1}{2}$.

\subsection{The structure of the Hecke algebra}
\label{subs_structure}

We now fix the setting for the rest of the paper. We retain the notation $\rho, \sigma$ from the previous subsection, and consider the cuspidal component $\mathfrak{s}$ attached to the representation
\[
\underbrace{\rho \otimes \dotsb \otimes \rho}_{n \text{ times}}\ \otimes\ \sigma
\]
of the Levi subgroup $M = \GL_k(F) \times \dotsb \times \GL_k(F) \times \SO(2n_0+1)$ in  $\SO(2N+1)$, where $N = nk+n_0$. In the rest of the paper, we restrict our attention to cuspidal components of the above form. This does not present a significant loss of generality, since the Hecke algebra of a general cuspidal component is the product of algebras corresponding to components described above. To simplify notation, we set $\H = \H_\mathfrak{s}$.

The structure of the Hecke algebra $\H$ has been completely described by Heiermann \cite{heiermann2010parametres,heiermann2011operateurs}. In his work, Heiermann shows that $\H$ is a Hecke algebra with parameters (the type of the algebra and the parameters depending on the specifics of the given case). When the component in question is of the form described above, we have three distinct cases, which we now summarize. For basic definitions and results on Hecke algebras with parameters, we refer to the work of Lusztig \cite{lusztig1989affine}. 

In what follows, we let $t$ denote the order of the (finite) group $\{\chi \in X^\text{nr}(M): \rho \otimes \chi \cong \rho\}$. In all three cases, the commutative algebra $\A$ (see \S \ref{subs_unramified}) is a subalgebra of $\H$. In the present setting, the rank of the free module $M^\sigma/M^\circ$ is equal to $n$. We can thus identify $\A = \mathbb{C}[M^\sigma/M^\circ]$ with the algebra of Laurent polynomials $\mathbb{C}[X_1^\pm,\dotsc,X_n^\pm]$. We fix this isomorphism explicitly: For $i = 1, \dotsc, n$, let $h_i$ be the element of $M$ which is equal to $\text{diag}(\varpi,1,\dotsc,1)$ on the $i$-th $\GL$ factor, and equal to the identity elsewhere. Then $X_i = b_{h_i}^t$. The three cases are
\begin{enumerate}[(i)]
\item No representation of the form $\rho\otimes \chi$ with $\chi \in X^\text{nr}(M)$ is self-dual. 

In this case, the algebra $\H$ is described by an affine Coxeter diagram of type $\tilde{A}_{n-1}$ with equal parameters $t$. In other words, it is isomorphic to the algebra $\H_n$ described in \cite{chan2019bernstein}: there are elements $T_1, \dotsc, T_{n-1}$ which satisfy the quadratic relation
\[
(T_i + 1)(T_i - q^t) = 0, \quad i = 1,\dotsc, n-1
\]
and commutation relations
\[
T_i f  - f^{s_i}T_i = (q^t-1) \frac{f - f^{s_i}}{1 - X_{i+1}/X_i}, \quad i = 1,\dotsc, n-1,
\]
where $f^{s_i}$ is obtained from $f \in \A$ by swapping $X_i$ and $X_{i+1}$.
\end{enumerate}
In the two remaining cases there is an unramified character $\chi$ of $M$ such that $\rho \otimes \chi$ is self-dual. Without loss of generality, we may assume that $\rho$ is self-dual. Then, up to isomorphism, there is a unique representation of the form $\rho \otimes \chi$ which is also self-dual; we denote it by $\rho^-$. We set $\alpha = \alpha_\rho$, $\beta = \alpha_{\rho^-}$ (see \S\ref{subs_cusp} for notation). Since the situation is symmetric, we may (and will) assume that $\alpha \geq \beta$. We then have the following two cases:
\begin{itemize}
\item[(ii)] $\alpha = \beta = 0$.

In this case, $\H$ is described by an affine Coxeter diagram of type $\tilde{C}_n$:
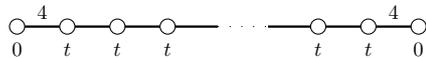
\[
\dynkin[labels={0,t,t,t,t,t,0}, extended, Coxeter, Kac, root/.style={fill=white}, root radius = 1mm, edge/.style={line width = 0.3mm}]C{}
\]
The nodes correspond to operators $T_0, \dotsc, T_n$ which satisfy the quadratic relations
\[
T_0^2 = 1, \quad  T_n^2 = 1, \quad (T_i+1)(T_i-q^t) = 0\quad  \text{for }i = 1 \dotsc, n-1,
\]
and the braid relations as prescribed by the diagram. The commutation relations for $T_i, i=1,\dotsc, n-1$ are the same as in case (i), whereas $T_n$ satisfies
\[
fT_n - T_nf^\vee = 0
\]
with $f^\vee(X_1,\dotsc,X_{n-1},X_n) = f(X_1,\dotsc,X_{n-1},1/X_n)$.

\item[(iii)] $\alpha > 0$.

In this case, $\H$ is described by an affine Coxeter diagram of type $\tilde{C}_n$:
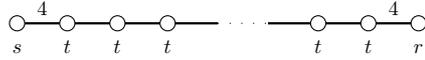
\[
\dynkin[labels={s,t,t,t,t,t,r}, extended, Coxeter, Kac, root/.style={fill=white}, root radius = 1mm, edge/.style={line width = 0.3mm}]C{}
\]
Here $s = t(\alpha-\beta)$ and $r = t(\alpha + \beta)$.
Again, the nodes correspond to operators $T_0, \dotsc, T_n$ which satisfy quadratic relations analogous to those in (ii), along with the braid relations.
The commutation relations for $T_i, i=1,\dotsc, n-1$ are the same as in case (i), whereas $T_n$ satisfies
\[
fT_n -  T_nf^\vee  = \left((q^r-1) + \frac{1}{X_n}(\sqrt{q}^{r+s}-\sqrt{q}^{r-s})\right)\frac{f - f^\vee}{1-1/X_n^2}.
\]

\end{itemize}

Cases (i)--(iii) correspond to the cases (I)--(III) listed in \cite{heiermann2010parametres}, Section 3.1. The above results are collected in Section 3.4 of \cite{heiermann2010parametres}. A detailed construction of the operators $T_i$ (starting from standard intertwining operators) is the subject matter of \cite{heiermann2011operateurs}; we do not need the details here, except in a special case discussed in the final part of Section \ref{subs_GG}. To facilitate the comparison of the above summary to the works of Heiermann \cite{heiermann2010parametres,heiermann2011operateurs}, we point out the way in which our summary deviates from them:
\begin{rem}
\begin{enumerate}[a)]
\item The explicit isomorphism $\mathbb{C}[M^\sigma/M^\circ] \cong \mathbb{C}[X_1^\pm,\dotsc,X_n^\pm]$ we use is different than the one used in \cite{heiermann2010parametres}; there, Heiermann sets $X_i = b_{h_ih_{i+1}^{-1}}^t$ for $i = 1, \dotsc, n-1$.
\item The operator $T_0$ which appears in cases (ii) and (iii) above is not needed to describe $\H$, and is therefore not used in \cite{heiermann2010parametres} and \cite{heiermann2011operateurs}. To be precise, the Hecke algebra is generated over $\A$ by the operators $T_1,\dotsc,T_n$ and determined by the quadratic and braid relations they satisfy, along with the commutation relations listed above. Each of the operators $T_1,\dotsc, T_n$ corresponds to a simple reflection in the Weyl group, whereas the operator $T_0$ corresponds to the reflection given by the (in this case, unique) minimal element of the root system---see \cite[\S 1.4]{lusztig1989affine}. In fact, we define $T_0$ by setting
\[
T_0 = \sqrt{q}^{s+t(n-1)+r}X_1T_w^{-1},
\]
where $T_w = T_1\dotsb T_{n-1}T_nT_{n-1}\dotsb T_1$---see \cite[\S 2.8,3.3]{lusztig1989affine}. We use $T_0$ out of convenience, as it allows a more symmetric description of certain $\H$-modules.
\item The description of $\H$ in Case (ii) differs from the one given in \cite{heiermann2010parametres}, which views $T_n$ as the non-trivial element of the $R$-group. However, one can verify that the description we use is equivalent. With our description, (ii) can be viewed as a special case of (iii) (with $r = s = 0$); however, since our results in (ii) require additional analysis, we still state the two cases separately.
\end{enumerate}
\end{rem}


\subsection{Generic representations}
We recall only the most basic facts here; a general reference is e.g.~\cite{shahidi1981certain}. Again, we focus on the case $G = \SO(2N+1,F)$.

Let $B = TU$ denote the standard Borel subgroup of $\SO(2N+1,F)$, i.e.~the group consisting of all upper-triangular matrices in $\SO(2N+1,F)$. Here $T$ denotes the maximal torus (diagonal matrices), and $U$ is the group of all unipotent upper-triangular matrices. We fix a non-degenerate character $\psi$ of $U$. We say that a representation $(\pi,V)$ of $SO(2N+1,F)$ is $\psi$-generic if there exists a so-called Whittaker functional---that is, a linear functional $L: V\to \mathbb{C}$ such that
\[
L(\pi(u)v) = \psi(u)L(v), \quad \forall u\in U, v\in V.
\]
The key fact we use throughout is that the space of Whittaker functionals is at most one-dimensional. If $\pi$ is an irreducible representation with a Whittaker functional $L$, one can consider the space $\mathcal{V}$ of all functions $f_v:G \to \mathbb{C}$, where $f_v(g) =L(\pi(g)v)$; we let $G$ act on this space by right translations. Then $v\mapsto f_v$ is a $G$-isomorphism; we say that $\mathcal{V}$ is the Whittaker model for $\pi$.

Now let $P=MN$ be a parabolic subgroup of $G$. If $\sigma$ is an irreducible generic representation of $M$, then one can construct a Whittaker functional on $i_P^G\sigma$ (see \cite[Proposition 3.1]{shahidi1981certain} and equation \eqref{eq_heredity} below); in other words, the induced representation is $\psi$-generic as well. We use this fact later, in Section \ref{subs_GG}.

\section{The Gelfand--Graev representation}

Let $U$ be the group consisting of all unipotent upper-triangular matrices in $\SO(2N+1)$. We fix a non-degenerate character $\psi : U \to \mathbb{C}^\times$ and consider the compactly induced representation $\cind_U^G(\psi)$. This is the so-called Gelfand--Graev representation. It is the “universal” $\psi$-generic representation: every $\psi$-generic representation of $G$ appears as a quotient (with multiplicity one) of $\cind_U^G(\psi)$. Our goal is to determine the structure of the Gelfand--Graev representation viewed as an $\H$-module.

We let $\Pi$ denote the $\H$-module $(\cind_U^G(\psi))_\mathfrak{s}$, where $\mathfrak{s}$ is the cuspidal component we fixed in Section \ref{subs_structure}. Our investigation of the structure of $\Pi$ begins with the following result.
\begin{prop}
As $\A$-modules, we have $\Pi \cong \A$.
\end{prop}
\begin{proof}
The $\H$-module $\Pi$ is given by $\Hom_G(\Gamma_\mathfrak{s},\cind_U^G(\psi))$, where $\Gamma_\mathfrak{s} =  i_P^G(\cind_{M^\circ}^M(\sigma_0))$. 
Recall that $\sigma_0$ was taken to be an arbitrary irreducible constituent of $\sigma|_{M^\circ}$. However, having now fixed the Whittaker datum for $M$ (and thus for $M^\circ$), there exists a unique irreducible summand of $\sigma|_{M^\circ}$ which is $\psi$-generic. Thus, from now on, we assume $\sigma_0$ is this unique generic constituent of $\sigma|_{M^\circ}$.

To view $\Pi$ as an $\A= \End_M(\cind_{M^\circ}^M(\sigma_0))$-module, we use the Bernstein version of Frobenius reciprocity:
\[
\Pi = \Hom_G( i_P^G(\cind_{M^\circ}^M(\sigma_0)),\cind_U^G(\psi)) = \Hom_M( \cind_{M^\circ}^M(\sigma_0), r_{\overline{N}}(\cind_U^G(\psi)));
\]
here $r_{\overline{N}}$ denotes the Jacquet functor with respect to $\overline{P}=M\overline{N}$, the parabolic opposite to $P$.

We now use the fact that $ r_{\overline{N}}(\cind_U^G(\psi))$ is isomorphic to the Gelfand--Graev representation of $M$, $\cind_{U\cap M}^M(\psi)$ (see \cite[\S 2.2]{bushnell2003generalized}). Furthermore, with the above choice of $\sigma_0$, the representation $\cind_{M^\circ}^M(\sigma_0)$ is precisely the sum of all maximal $({}^m\sigma_0)$-isotypic components of $\cind_{U\cap M}^M(\psi)$, where ${}^m\sigma_0$ ranges over the set of all $M$-conjugates of $\sigma_0$. Indeed, $\cind_{U\cap M}^M(\psi)$ is itself induced from the Gelfand--Graev representation of $M^\circ$, $\cind_{U\cap M}^{M^{\circ}}(\psi)$. Since $\sigma_0$ appears with multiplicity one, and no other $m$-conjugate of $\sigma_0$ is generic, we have $\cind_{U\cap M}^{M^{\circ}}(\psi) \cong \sigma_0 \oplus \sigma_0^\bot$, where $\sigma_0^\bot$ is a representation which contains no $M$-conjugate of $\sigma_0$. Inducing to $M$ we get $\cind_{U\cap M}^M(\psi) = \cind_{M^\circ}^M(\sigma_0) \oplus \cind_{M^\circ}^M(\sigma_0^\bot)$, which proves the above claim about isotypic components. Thus, viewed as an $\A$-module, $\Pi$ is isomorphic to
\begin{align*}
\Hom_M( \cind_{M^\circ}^M(\sigma_0), r_{\overline{N}}(\cind_U^G(\psi))) &= \Hom_M( \cind_{M^\circ}^M(\sigma_0), \cind_{U\cap M}^M(\psi))\\
& = \Hom_M( \cind_{M^\circ}^M(\sigma_0), \cind_{M^\circ}^M(\sigma_0)) = \A. 
\end{align*}
%
\end{proof}

The above result suggests the following approach to determining the $\H$-module structure of $\Pi$: First, we find all possible $\H$-module structures on $\A$. After that, we need only determine which one of those structures describes $\Pi$. In the following subsection, we compute the possible $\H$-structures on $\A$.

\subsection{$\H$-module structures on $\mathcal{A}$}
\label{subs_H_on_A}
First, assume that we are in Case (i) (see \S \ref{subs_structure}). Then the situation is precisely the one treated in \cite{chan2019bernstein}, and the possible $\H$-module structures on $\A$ are determined in Section 2.2 there. We have the following:
\begin{prop}[Case (i)]
\label{prop_structure_case_i} Let $\Pi$ be a $\H$-module which is isomorphic to $\A$ as an $\A$-module. Then $\Pi \cong \H \otimes_{\H_{S_n}} \epsilon$, where $\epsilon$ is a $1$-dimensional representation of $\H_{S_n}$.
\end{prop}
\noindent Here $\H_{S_n}$ denotes the finite-dimensional algebra generated by $T_1,\dotsc,T_{n-1}$; we have $\H = \A \otimes_\mathbb{C} \H_{S_n}$. Furthermore, $\H_{S_n}$ has precisely two one-dimensional representations:
\begin{align*}
\epsilon_{-1}:  T_i \mapsto -1 \quad &\text{ for } i = 1, \dotsc, n-1; \quad  \text{and}\\
\epsilon_{q^t}: T_i \mapsto q^t \quad &\text{ for } i = 1, \dotsc, n-1.
\end{align*}

We now treat Cases (ii) and (iii), simultaneously. Recall that in these cases the algebra $\H$ is described by an affine Coxeter diagram of type $\tilde{C}_n$. We let $\H_0$ and $\H_n$ denote the algebra obtained by removing the vertices which correspond to $T_0$ and $T_n$, respectively. In other words, $\H_0$ is generated by $T_1,\dotsc,T_n$ as an $\A$-algebra, whereas $\H_n$ is generated by $T_0,\dotsc,T_{n-1}$. Note that we have $\H = \A \otimes_{\mathbb{C}}\H_n = \A \otimes_{\mathbb{C}}\H_0$.
We now prove the following result.

\begin{prop}[Cases (ii) and (iii)]
\label{prop_HstructureA} Let $\Pi$ be a $\H$-module which is isomorphic to $\A$ as an $\A$-module. Then
\[
\Pi \cong \H \otimes_{\H_{0}} \epsilon_0 \quad \text{or} \quad \Pi \cong \H \otimes_{\H_{n}} \epsilon_n.
\]
Here $\epsilon_0$ (resp.~$\epsilon_n$) is a $1$-dimensional representation of $\H_{0}$ (resp.~$\H_{n}$).
\end{prop}

\begin{proof}
We first restrict our attention to the subalgebra generated by $T_1,\dotsc, T_{n-1}$, which is contained in both $\H_0$ and $\H_n$. This is precisely the algebra $\H_{S_n}$ discussed in \cite{chan2019bernstein}. The possible $\H_{S_n}$-structures on $\A$ are determined in \S 2.2 there. To summarize the relevant results, there exists an invertible element $g_0\in \A$ on which the operators $T_1,\dots,T_n$ act by the same scalar, either $q^t$ or $-1$.

We now determine how $T_0$ and $T_n$ act on $g_0$. Since $g_0$ is invertible, we have $T_ng_0 = fg_0$ for some $f\in \A$. Recall that $T_n$ satisfies the quadratic relation
\[
T_n^2 = (q^r-1)T_n + q^r
\]
as well as the commutation relation
\[
T_nf- f^\vee T_n = \left((q^r-1) + \frac{1}{X_n}(\sqrt{q}^{r+s}-\sqrt{q}^{r-s})\right)\frac{f-f^\vee}{1-1/X_n^2}.
\]
Here, and throughout the proof, we let $r=s=0$ if we are considering Case (ii).
Recall that $f^\vee$ denotes the function $f^\vee(X_1,\dotsc,X_n) = f(X_1,\dotsc,X_{n-1},\frac{1}{X_n})$. Using the above and comparing the two sides of $T_n^2g_0 = (q^r-1)T_ng_0 + q^rg_0$, we get
\[
ff^\vee = (q^r-1)\frac{X_nf^\vee-\frac{1}{X_n}f}{X_n-\frac{1}{X_n}} - (\sqrt{q}^{r+s}-\sqrt{q}^{r-s})\frac{f-f^\vee}{X_n-\frac{1}{X_n}} + q^r.
\]
To simplify notation, we now set $b = q^r-1$, $c = (\sqrt{q}^{r+s}-\sqrt{q}^{r-s})$. We also temporarily drop the index $n$, writing $X$ instead of $X_n$. Clearing out the denominators, we rearrange the above equation into
\[
\label{eq_ff}\tag{\textasteriskcentered}
(X^2-1)ff^\vee = b(X^2f^\vee - f) - c(Xf- Xf^\vee) + q^r(X^2-1).
\]
Our first goal is to find the possible solutions $f\in \A$ of this equation.
\begin{lem}
\label{lem_sol}
The above equation has the following solutions:
\begin{align*}
f &= b + cX^{-1} + bX^{-2} + \dotsc + cX^{1-2d} +q^tX^{-2d}, && d \in \mathbb{Z}_{>0}\tag{i}\\
f &= b + cX^{-1} + bX^{-2} + \dotsc + cX^{1-2d} - X^{-2d}, && d \in \mathbb{Z}_{>0}\tag{ii}\\
f &= b + cX^{-1} + bX^{-2} + \dotsc + bX^{-2d} \pm \sqrt{q}^{r\pm s}X^{-2d-1}, && d \in \mathbb{Z}_{\geq 0}\tag{iii}\\
f &= \mp\sqrt{q}^{r\pm s}X^{2d+1}-bX^{2d} -cX^{2d-1}-\dotsb-cX, && d \in \mathbb{Z}_{\geq 0}\tag{iv}\\
f &= -q^rX^{2d} - cX^{2d-1} - \dotsc - bX^2-cX, && d \in \mathbb{Z}_{>0}\tag{v}\\ 
f &= X^{2d} - cX^{2d-1} - \dotsc - bX^2-cX, && d \in \mathbb{Z}_{>0}\tag{vi}.
\end{align*}
along with the constant solutions $f = q^t$ and $f=-1$.
\end{lem}
\begin{proof}
Each $f\in \A$ can be written as
\[
\label{eq_f}\tag{$\dagger$}
f = a_kX^k + a_{k-1}X^{k-1} + \dotsb + a_0 + a_{-1}X^{-1} + \dotsb + a_{-l}X^{-l}
\]
for some functions $a_{-l},\dotsc,a_k \in \mathbb{C}[X_1^\pm, \dotsc, X_{n-1}^\pm]$, with $a_k, a_{-l}\neq 0$. We write $\text{maxdeg}(f)$ for $k$ and $\text{mindeg}(f)$ for $-l$. Now let $f$ be a solution of \eqref{eq_ff}.
We begin our analysis of \eqref{eq_ff} by solving some special cases. We claim the following:
\begin{equation}
\label{eq_base}
\begin{matrix}
\text{If} & f =a_0, &\text{then}  & a_0 = q^r \text{ or } a_0 = -1.\\
\text{If} &  f =a_1X, &\text{then} & a_1 = \mp\sqrt{q}^{r\pm s}.\\
\text{If} & f = a_0 + a_{-1}X^{-1} \text{ and } a_{-1}\neq 0, &\text{ then } & a_0 = b \text{ and } a_{-1} = \pm\sqrt{q}^{r\pm s}.
\end{matrix}
\end{equation}
To verify this, we first look at solutions $f = a_0$. In this case the equation \eqref{eq_ff} reduces to $a_0^2 = ba_0 + q^r$. This equation has two constant solutions, $a_0 = -1$ and $a_0 = q^r$. These are also the only solutions---all the solutions are contained in the field of rational functions $\mathbb{C}(X_1,\dotsc,X_{n-1})$, and any quadratic equation has at most two solutions over a given field. When $f(X) = a_1X$, the equation becomes $a_1^2 + a_1c - q^r = 0$. Again, the only two solutions of this equation are the constant ones: $a_1 = \mp\sqrt{q}^{r\pm s}$. Finally, when $f = a_0 +a_{-1}X^{-1}$, the equation reduces to the following system:
\[
a_1b = a_1a_0\quad \text{and}\quad a_0^2 + a_{-1}^2 = a_0b + a_{-1}c + q^r.
\]
Since we are assuming that $a_1\neq 0$, the first equation gives us $a_0 = b$, and then the second becomes $a_{-1}^2 - ca_{-1} - q^r$. Again, we have two solutions: $a_{-1} = \pm \sqrt{q}^{r\pm s}$.

Next, when $f$ is a solution of \eqref{eq_ff} given by \eqref{eq_f}, we observe:
\begin{equation}
\label{eq_kl}
k \text{ and } l \text{ cannot both be positive.}
\end{equation} 
Indeed, let LHS and RHS denote the left-hand side and the right-hand side of \eqref{eq_ff}, respectively. We then have $\text{maxdeg}(LHS) = k+l+2$, whereas $\text{maxdeg}(RHS) \leq \max\{l+2,k+1,2\}$. Therefore equality of degrees cannot be achieved unless $k \leq 0$ or $l\leq 0$. In fact, the same argument gives us a slightly stronger statement in one case:
\begin{equation}
\label{eq_a0}
\text{If } k \gneqq 0  \text{ then } a_0 = 0.
\end{equation} 
Finally, we make use of the following fact, which is readily verified by direct computation:
\begin{equation}
\label{eq_symmetry}
\begin{gathered}
\text{For any positive integer } d, f \text{ is a solution of \eqref{eq_ff} if and only if }\\
 X^{2d}f-R_d \text{ is also a solution.}
\end{gathered}
\end{equation}
Here $R_d = \dfrac{bX^2+cX}{X^2-1}(X^{2d}-1) = bX^{2d} + cX^{2d-1}+\dotsc+bX^2+cX$.

\bigskip

We are now ready to find all the solutions. By \eqref{eq_kl}, any solution of $f$ either contains only positive powers of $X$, or only non-positive. We therefore  consider two separate cases:

\medskip

\noindent \textbf{Case A:} $f$ has only non-positive powers, i.e.~$f = a_0 + a_{-1}X^{-1} + \dotsb + a_{-l}X^{-l}$.

\medskip

Let $d = \lfloor l/2 \rfloor$. We use \eqref{eq_symmetry} and look at another solution, $g = X^{2d}f - R_d$.

We first assume $l=2d$ is even. In this case $g$ only has non-negative powers of $X$, but it has a non-zero constant term, $a_{-l}$. Therefore \eqref{eq_a0} shows that the coefficients next to the positive powers must be zero: $a_0-b = a_{-1}-c= \dotsb = a_{-l+1}-c = 0$. Now \eqref{eq_base} shows that there are only two possibilities for the constant term: $a_{-l} = q^t$ or $a_{-l} = -1$. We thus get two solutions:
\[
f = b + cX^{-1} + bX^{-2} + \dotsc + cX^{1-2d} +q^tX^{-2d} \text{ } \text{ and }\text{ } f = b + cX^{-1} + bX^{-2} + \dotsc + cX^{1-2d} - X^{-2d}.
\]
Next, assume that $l=2d+1$ is odd. Now $g$ has a non-zero coefficient (i.e.~$a_{-l}$) next to $X^{-1}$, so by \eqref{eq_kl} the coefficients next to positive powers must be equal to $0$. This gives us $a_0 = b, a_{-1} = c, \dotsc, a_{2-l} = c$. Furthermore, $g$ is thus of the form $a_{1-l} + a_{-l}X^{-1}$, so we can read off the coefficients $a_{1-l}$ and $a_{-l}$ from \eqref{eq_base}. We thus arrive at two more solutions:
\[
f = b + cX^{-1} + bX^{-2} + \dotsc + bX^{-2d} \pm \sqrt{q}^{r\pm s}X^{-2d-1}.
\]

\noindent \textbf{Case B:} $f$ only has positive powers, i.e.~$f = a_kX^k +\dotsc + a_1X$.

\medskip

This time, we set $d = \lfloor k/2 \rfloor$ and use \eqref{eq_symmetry} to obtain the solution $g = \dfrac{1}{X^{2d}}(f+R_d)$. 

First, assume that $k = 2d+1$ is odd. Then $g$ has a non-zero coefficient (i.e.~$a_k$) next to $X$, so \eqref{eq_kl} and \eqref{eq_a0} imply that all the lower coefficients are zero. This immediately gives us $a_1 = -c, a_2 = -b, \dotsc, a_{2d} = -b$. Furthermore, we have $g = a_kX$, so \eqref{eq_base} shows that we have two possibilities for $a_k$. We therefore get two solutions:
\[
f = \mp\sqrt{q}^{r\pm s}X^{2d+1}-bX^{2d} -cX^{2d-1}-\dotsb-cX.
\]
Finally, assume that $k=2d$ is even. First, if $k > 2$, consider another solution $g' = X^{2-2d}(f+R_{2d-2})$. Now $g'$ has a non-zero coefficient (i.e.~$a_k$) next to $X^2$, so the coefficient next to non-positive powers of $X$ have to be $0$ by \eqref{eq_kl}, \eqref{eq_a0}. This gives us $a_1 = -c, a_2 = -b, \dotsc, a_{2d-2}= -b$. In particular, this shows that $g = (a_k+b) + (a_{k-1}+c)X^{-1}$. Since $a_k + b \neq b$ (i.e.~$a_k\neq 0$), \eqref{eq_base} shows that we have only two possibilities:
\[
a_{k-1} +c = 0,\quad a_k+b \in \{q^r, -1\}.
\]
In other words, $a_{k-1} = -c$ and $a_k \in \{-q^r, 1\}$. We thus get the remaining solutions,
\[
f = -q^rX^{2d} - cX^{2d-1} - \dotsc - bX^2-cX \quad \text{and}\quad f = X^{2d} - cX^{2d-1} - \dotsc - bX^2-cX.
\]
\end{proof}

We continue the proof of Proposition \ref{prop_HstructureA}.
We have just proved that $T_ng_0 = fg_0$ where $f\in \A$ is one of the elements listed in Lemma \ref{lem_sol}. First, assume that $f$ is one of the constant solutions, i.e.~$f=-1$ or $f=q^r$. Then $g_0$ is an invertible element of $\A$ on which $T_1, \dotsc, T_{n-1}, T_n$ all act as scalars. In other words, we have a one-dimensional representation $\epsilon_0$ of the algebra $\H_0$. Since $\H = \A \otimes_{\mathbb{C}} \H_0$, it follows that the corresponding $\H$-module structure on $\A$ is isomorphic to
\[
\H\otimes_{ \H_0} \epsilon_0.
\]
Now, if $f$ is of type (i) or (ii) listed in the statement of Lemma \ref{lem_sol}, set \[
g_1 = (X_1X_2\cdot \dotsb \cdot X_n)^{-d}g_0.
\]
Since $(X_1X_2\cdot \dotsb \cdot X_n)^{-d}$ commutes with $T_1,\dotsc,T_{n-1}$, $g_1$ is still an eigenvector for each of these operators. We claim that $g_1$ is an eigenvector for $T_n$ as well. Indeed, using the appropriate commutation relation and the fact that $T_n$ commutes with $X_1, \dotsc, X_{n-1}$, we get
\begin{align*}
T_ng_1 &= (X_1X_2\cdot \dotsb \cdot X_{n-1})^{-d}\cdot T_nX_n^{-d}g_0\\
 &= (X_1X_2\cdot \dotsb \cdot X_{n-1})^{-d}\left(X_n^dT_n + \frac{bX_n+c}{X_n^2-1}(X_n^{-d}-X_n^d)\right)g_0\\
&= (X_1X_2\cdot \dotsb \cdot X_{n-1})^{-d}\left(X_n^df + \frac{bX_n+c}{X_n^2-1}(X_n^{-d}-X_n^d)\right)g_0\\
&= (X_1X_2\cdot \dotsb \cdot X_{n-1})^{-d}\left(X_n^{2d}f - \frac{bX_n+c}{X_n^2-1}(X_n^{2d}-1)\right)g_1.
\end{align*}
Simplifying the expression in the parentheses, we obtain $\lambda X_n^{-d}$, so that $T_ng_1 = \lambda g_1$, where $\lambda = q^t$, resp.~$-1$ when $f$ is of type (i), resp.~(ii).
We have thus once more found a common eigenvector for $T_1,\dotsc,T_{n-1},T_n$. Again, we deduce that the corresponding $\H$-module structure is isomorphic to $\H\otimes_{ \H_0} \epsilon_0$, where $\epsilon_0$ is a one-dimensional representation of $\H_0$.

When $f$ is of type (v) or (vi), we use the same argument and arrive at the same conclusion. The only difference in this case is that we have to set $g_1 = (X_1X_2\cdot \dotsb \cdot X_n)^{d}g_0$ in order to obtain a common eigenvector for $T_1,\dotsc,T_{n-1},T_n$.

In the remaining cases---that is, when $f$ is of type (iii) or (iv)---we cannot find such an eigenvector, but we claim that we can find an invertible $g_1 \in A$ which is a common eigenvector for $T_0, T_1, \dotsc, T_{n-1}$. Just like in the previous cases, this will imply that the $\H$-structure on $\A$ is isomorphic to $\H\otimes_{ \H_n} \epsilon_n$ for some one-dimensional representation $\epsilon_n$ of $\H_n$.

If $T_ng_0 = fg_0$ with $f$ of type (iii), we set $g_1 = (X_1X_2\cdot \dotsb \cdot X_n)^{-d}g_0$. If $f$ is of type (iv), let $g_1 = (X_1X_2\cdot \dotsb \cdot X_n)^{d+1}g_0$. In both cases, $g_1$ is an eigenvector for $T_1,\dotsc, T_{n-1}$ and a computation analogous to the one we carried out in for cases (i) and (ii) shows that we have
\[
T_ng_1 = (b\pm \sqrt{q}^{r\pm s}X_n^{-1})g_1.
\]
The following lemma then shows that $g_1$ is also an eigenvector $T_0$ and thus concudes the proof of Proposition \ref{prop_HstructureA}.

\begin{lem}
Let $g$ be an invertible element of $\A$ which is an eigenvector for $T_1, \dotsc, T_{n-1}$ and such that $T_ng = (b\pm \sqrt{q}^{r\pm s}X_n^{-1})g$.
Then $g$ is also an eigenvector for $T_0$.
\end{lem}
\begin{proof}
Recall that $T_0 = \sqrt{q}^{s+2(n-1)t+r}X_1T_w^{-1}$, with $T_w = T_1\dotsm T_{n-1}T_nT_{n-1}\dotsm T_1$. In both cases, all the operators $T_1,\dotsc,T_{n-1}$ act by the same scalar $\lambda \in \{-1,q^t\}$. We therefore have 
\[
T_0g =  \sqrt{q}^{s+2(n-1)t+r} \lambda^{-(n-1)} X_1 T_1^{-1}\cdot\dotsb\cdot T_{n-1}^{-1}T_n^{-1}g.
\]
We now recall that $T_n^{-1} = \frac{1}{q^r}(T_n-b)$; this follows from the quadratic relation for $T_n$. Therefore, by the assumption in the statement of the lemma, $T_n^{-1}g = \pm \sqrt{q}^{\pm s -r}X_n^{-1}$. Thus
\begin{equation}
\label{eq_T0}
T_0g = \mu\cdot \lambda^{-(n-1)} \cdot \sqrt{q}^{2(n-1)t}X_1 T_1^{-1}\cdot\dotsb\cdot T_{n-1}^{-1}X_n^{-1}g,
\end{equation}
with $\mu \in \{-1, q^s\}$. Finally, it remains to notice that for every $i = 1, \dotsc, n-1$ we have
\begin{equation}
\label{eq_com}
T_i^{-1}X_{i+1}^{-1} = \frac{1}{q^t}X_{i}^{-1}T_{i}.
\end{equation}
Indeed, from the quadratic relation we have $T_i^{-1}= \frac{1}{q^t}(T_i - (q^t-1))$. Combining this with the commutation relation for $T_i$, we get \eqref{eq_com}. Successively applying \eqref{eq_com} to \eqref{eq_T0} (and taking into account that each $T_i$ acts on $g$ by $\lambda$), we get
\[
T_0g = \mu g,
\]
which we needed to prove. Notice that the possible eigenvalues are precisely the zeros of $(x-q^s)(x+1) = 0$, the quadratic equation satisfied by $T_0$.
\end{proof}
The above lemma shows that in cases (iii) and (iv) we have an invertible element $g_1 \in \A$ which is a common eigenvector for $T_0,T_1,\dotsc, T_n$. Consequently, the $\H$-module structure on $\A$ is given by $\H\otimes_{ \H_n} \epsilon_n$ for some one-dimensional representation $\epsilon_n$ of $\H_n$. This concludes the proof of Proposition \ref{prop_HstructureA}.
\end{proof}

In view of Proposition \ref{prop_HstructureA}, there are eight candidates for the $\H$-structure (four, if $n=1$): First, we may take the tensor product over $\H_0$ or $\H_n$; after that, there are four $1$-dimensional representations of $\H_0$ (resp.~$\H_n$) to choose from. To verify this, note that the braid relations imply that---in any $1$-dimensional representation---the operators $T_1, \dotsc, T_{n-1}$ act by the same scalar, which has to be a zero of the quadratic relation satisfied by $T_i$: $
(x-q^t)(x+1) = 0$. We therefore have two possibilities for the action of the operators $T_i$, and two additional possibilities (again, the zeroes of the quadratic relation) for $T_n$ (resp.~$T_0$). For example, the $1$-dimensional representations of $\H_0$ are given by
\begin{align*}
\epsilon_{-1,-1} &: \{T_n \mapsto -1, T_i \mapsto -1\} \quad  &\epsilon_{q^r,-1} : \{T_n \mapsto q^r, T_i \mapsto -1\}\\
\epsilon_{-1,q^t} &: \{T_n \mapsto -1, T_i \mapsto q^t\} \quad  &\epsilon_{q^r,q^t} : \{T_n \mapsto q^r, T_i \mapsto q^t\}.
\end{align*}

\subsection{The Gelfand--Graev module}
\label{subs_GG}
To complete the analysis of the Gelfand--Graev representation, we need to determine which of the $\H$-module structures from the previous section is isomorphic to $\Pi = (\cind_U^G(\psi))_\mathfrak{s}$. We consider the cases (i)--(iii) separately.

\medskip

\noindent \textbf{Case (i).} Let $\delta$ be the unique irreducible subrepresentation of $\rho\nu^{\frac{n-1}{2}} \times \rho\nu^{\frac{n-3}{2}} \times \dotsb \times \rho\nu^{\frac{1-n}{2}}$. Then $\pi = \delta \rtimes \sigma$ is an irreducible generic representation. The corresponding $\H$-module is $1$-dimensional: by the Bernstein version of Frobenius reciprocity, we have
\begin{equation}
\label{eq_Frob_1}
\begin{aligned}
\Hom_G(\Gamma_\mathfrak{s}, \pi) =  \Hom_M(\cind_{M^\circ}^M(\rho \otimes \dotsb \otimes \rho \otimes \sigma), &\ \nu^{\frac{1-n}{2}}\rho \otimes \dotsb \otimes \nu^{\frac{n-1}{2}}\rho\otimes \sigma\\
\oplus  &\ \nu^{\frac{1-n}{2}}\rho^\vee \otimes \dotsb \otimes \nu^{\frac{n-1}{2}}\rho^\vee \otimes \sigma)
\end{aligned}
\end{equation}
Since $\rho^\vee$ is not an unramified twist of $\rho$ in this case, the above $\Hom$-space is only $1$-dimensional. By Proposition \ref{prop_structure_case_i}, $\Hom_G(\Gamma_\mathfrak{s}, \Pi)$ is isomorphic to either $\Pi \cong \H \otimes_{\H_{S_n}} \epsilon_{-1}$ or $\Pi \cong \H \otimes_{\H_{S_n}} \epsilon_{q^t}$. To determine which, we need only look at the action of $\H$ on the $1$-dimensional module $\pi$. We now need to examine the definition of the operators $T_i, i= 1,\dotsc, {n-1}$. In \cite{heiermann2011operateurs}, $T_i$ is defined in \S 5.2 by the formula 
\begin{equation}
\label{eq_T_i}
T_i = R_i + (q^t-1)\dfrac{X_i/X_{i+1}}{X_{i}/X_{i+1}-1}.
\end{equation}
The intertwining operator $R_i$ has a pole at $0$, and a zero at the point of reducibility---see \cite[\S 1.8]{heiermann2011operateurs}. Since $\nu^{\frac{3-n}{2}-i}\rho \times \nu^{\frac{3-n}{2}-i+1}\rho$ reduces, the operator $R_i$ acts by $0$ in this case. It therefore remains to determine the action of $X_i/X_{i+1}$. Equation \eqref{eq_Frob_1} shows that it suffices to determine the action of $X_i/X_{i+1}$ on
\[
\Hom_M(\cind_{M^\circ}^M(\rho \otimes \dotsb \otimes \rho \otimes \sigma),  \nu^{\frac{1-n}{2}}\rho \otimes \dotsb \otimes \nu^{\frac{n-1}{2}}\rho\otimes \sigma).
\]
Recalling the definition of $X_i$ (\S \ref{subs_structure}), we immediately see that $X_i/X_{i+1}$ acts by 
\[
\dfrac{(|\varpi|^{\frac{3-n}{2}-i})^t}{(|\varpi|^{\frac{3-n}{2}-i+1})^t} = \frac{q^{t(\frac{n-3}{2}+i)}}{q^{t(\frac{n-3}{2}+i-1)}} = q^t.
\]
This implies that $T_i$ also acts by $(q^t-1)\dfrac{q^t}{q^t-1} = q^t$. Since $\pi$ is a quotient of $\Pi$, we conclude that we must have $\Pi \cong \H \otimes_{\H_{S_n}} \epsilon_{q^t}$.

\bigskip

\noindent \textbf{Case (iii).} In this situation, the $\mathfrak{s}$-component of the Gelfand--Graev representation has two irreducible generic representations whose $\H$-module is one-dimensional. These are the two (generalized) Steinberg representations: $\pi$ and  $\pi'$, which are the unique irreducible subrepresentations of
\[
\nu^{\alpha+n-1} \times \dotsb \times \nu^{\alpha} \rtimes \sigma\quad  \text{and}\quad \nu^{\beta+n-1}\rho^- \times \dotsb \times \nu^{\beta}\rho^- \rtimes \sigma
\]
respectively. Recall that $\alpha$ (resp.~$\beta$) is the unique positive real number such that $\nu^\alpha\rho \rtimes \sigma$ (resp.~$\nu^\beta\rho^- \rtimes \sigma$) reduces (see \S\ref{subs_structure}). We now compare the action of the operators $T_0, \dotsc, T_n$ on these two representations---that is, on $\Hom_G(\Gamma_\mathfrak{s}, \pi)$ and $\Hom_G(\Gamma_\mathfrak{s}, \pi^-)$, where  $\Gamma_\mathfrak{s}$ is the projective generator defined in \S\ref{subs_Hecke}.

We start by analyzing the action on $\pi$. We first focus on $T_i, i = 1, \dotsc, n-1$. Again, $T_i$ is defined by \eqref{eq_T_i}, and once more the operator $R_i$ acts by $0$. By the Bernstein version of Frobenius reciprocity, we have
\[
\Hom_G(\Gamma_\mathfrak{s}, \pi) = \Hom_M(\cind_{M^\circ}^M(\rho \otimes \dotsb \otimes \rho \otimes \sigma), \nu^{-\alpha-{n+1}}\rho \otimes \dotsb \otimes \nu^{-\alpha}\rho\otimes \sigma).
\]
We immediately see that $X_i/X_{i+1}$ acts by 
\[
\dfrac{(|\varpi|^{-\alpha-n+i})^t}{(|\varpi|^{-\alpha-n+i+1})^t} = \frac{q^{t(\alpha+n-i)}}{q^{t(\alpha+n-i-1)}} = q^t.
\]
Again, this shows that $T_i$ acts by $(q^t-1)\dfrac{q^t}{q^t-1} = q^t$. For $T_n$ we have a similar formula:
\begin{equation}
\label{eq_T_n}
T_n = R_n + (q^r-1)\dfrac{X_n\left(X_n-\dfrac{q^{t\beta}-q^{t\alpha}}{q^r-1}\right)}{X_n^2-1}.
\end{equation}
Once more, $R_n$ acts by $0$, and $X_n$ acts by $(|\varpi|^{-\alpha})^t = q^{t\alpha}$. Recalling that $r=t(\alpha+\beta)$, we see that $T_n$ acts by $q^r$. Finally, since
\[
T_0 = \sqrt{q}^{r+2t(n-1)+s}X_1T_1^{-1}\cdot \dotsb \cdot T_{n-1}^{-1}T_n^{-1}T_{n-1}^{-1}\cdot \dotsb \cdot T_1^{-1},
\]
and since $X_1$ acts by $q^{(\alpha+n-1)t}$, we see that $T_0$ acts by $\dfrac{\sqrt{q}^{r+2t(n-1)+s}}{q^{2t(n-1)}\cdot q^r}q^{(\alpha+n-1)t} =q^s$.

We do the same with $\pi^-$. Again, $X_i/X_{i+1}$ acts by $q^t$ which shows that $T_i$ acts by $q^t$ as well. This time $X_n$ acts by $-q^{t\beta}$: recall that $\rho^{-} = \chi_0 \otimes \rho$ with $X_n(\chi_0) = -1$, so $X_n(\chi_0\nu^{-\beta})= -q^{t\beta}$. Repeating the above calculations we now see that $T_n$ acts by $q^r$, whereas $T_0$ acts by $-1$.

The above analysis allows us to single out the $\H$-module structure on $\Pi$. Since $T_0$ does not act by the same scalar on $\pi$ and $\pi^-$, we deduce that $\Pi = \H \otimes_{\H_0} \epsilon$ for some $1$-dimensional representation $\epsilon$ of $\H_0$. Now, since every $T_i$ ($i=1,\dotsc,n-1$) acts by $q^t$ and $T_n$ acts by $q^r$, we deduce that $\Pi = \H \otimes_{\H_0} \epsilon_{q^r,q^t}$ (see the end of \S \ref{subs_H_on_A} for notation).

\bigskip

\noindent \textbf{Case (ii)} The first part of our analysis remains the same as in Case (iii). The representation
\[
\nu^{n-1}\rho \times \nu^{n-2}\rho \times \dotsb \times \rho \rtimes \sigma 
\]
has two irreducible subrepresentations (both of which are in discrete series when $n>1$, and temepered when $n=1$), only one of which is generic. Denote the generic subrepresentation by $\pi$. Let $\pi^-$ denote the generic representation resulting from an analogous construction, when $\rho$ is replaced by $\rho^-$. Again, the $\H$-modules corresponding to $\pi$ and $\pi^-$ are $1$-dimensional,
and the same calculations we used in Case (iii) show that the operators $T_i$, $i = 1, \dotsc, n-1$ act by $q^t$. This leaves us four possible $\H$ structures to consider
\begin{equation}
\label{eq_case_ii_structures}
\begin{aligned}
\H \otimes_{\H_0} \epsilon_0, \quad \text{with} \quad &\epsilon_0(T_n) = \pm 1 \quad (\text{and }\epsilon_0(T_i) = q^t,i = 1, \dotsc, n-1); \quad \text{and}\\
\H \otimes_{\H_n} \epsilon_n, \quad \text{with} \quad &\epsilon_n(T_0) = \pm 1 \quad  (\text{and }\epsilon_n(T_i) = q^t, i = 1, \dotsc, n-1).
\end{aligned}
\end{equation}

So far, we have been able to view Case (ii) as a special instance of Case (iii) which occurs when $r=s=0$. However, to obtain an explicit description of the Gelfand--Graev module, we need more information than we used above in Case (iii).
The reason is that the standard intertwining operator $\chi\rho \rtimes \sigma \to \chi^{-1}\rho^\vee \rtimes \sigma$ no longer has a pole when $X_n(\chi) =\pm 1$. In Case (iii), the operator $R_n$ (see formula \eqref{eq_T_n})---which is constructed from the standard intertwining operator---vanishes at the point of reducibility, and the action of $T_n$ is determined by the action of the function
\[
(q^r-1)\dfrac{X_n\left(X_n-\dfrac{q^{t\beta}-q^{t\alpha}}{q^r-1}\right)}{X_n^2-1}
\]
used to remove the poles of $R_n$. In this case however, $R_n$ no longer vanishes and is regular at the point of reducibility; consequently, the above function does not appear in the construction and we have $T_n = R_n$. We know that this operator acts by $1$ or $-1$ on the $\H$-modules $\pi$ and $\pi^-$, but we still have a certain amount of freedom in our choices. Indeed, as one verifies easily, the operator $T_n' = (-1)^eX_n^fT_n$ (where $e \in \{0,1\}, f \in \mathbb{Z}$) satisfies the same relations as $T_n$. Therefore, we obtain the same Hecke algebra if we replace $T_n$ by $T_n'$, but the action of $T_n'$ on $\pi$ obviously differs from the action of $T_n$.

In fact, we know that $X_n$ acts on $\pi$ by $1$, and on $\pi^-$ by $-1$. Therefore $X_n^2$ acts by $1$ on both, so replacing $T_n$ by $X_n^2T_n$ does not affect our description of the Gelfand--Graev module. We thus have $4$ choices that affect the description ($e = 0$ or $1$; $f$ even or odd), and as we vary the four choices, the description of the Gelfand--Graev module varies through all four possibilities described in \eqref{eq_case_ii_structures} above.

This discussion shows that---to determine the action explicitly---we need to specify the choices appearing in the construction of the operator $R_n$. To do that, we make use of Whittaker models. In what follows, we specialize our discussion to the case $n=1$ to simplify notation (thus, the cuspidal representation which defines the component is $\rho \otimes \sigma$); the general case is completely analogous and follows from this one. We drop the subscripts and write $T,X$ instead of $T_n, X_n$. 

We fix a non-degenerate character $\psi$ of the unipotent radical $U$ of $G=\SO(2N+1,F)$. Let $V_\rho$ denote the space of the representation $\rho$, and let $\lambda$ be a $\psi$-Whittaker functional on $V_\rho$: $\lambda(\rho(u)v) = \psi(u)\lambda(v)$, for $v \in V_\rho$. Notice that $\lambda$ is then also a $\psi$-Whittaker functional for $\rho \otimes \chi$ for any unramified character $\chi \in \GL_k(F)$: we have
\[
\lambda((\chi\otimes \rho)(u)v) = \chi(u)\psi(u)\lambda(v) = \psi(u)\lambda(v),
\]
since $\det u = 1$ and thus $u \in \ker \chi$. Abusing notation, we also let $\lambda$ denote the $\psi$-Whittaker functional of $\rho \otimes \sigma$ (or $\chi\rho \otimes \sigma$ for any unramified $\chi$, as we have just shown). Following Proposition 3.1 of \cite{shahidi1981certain}, we now form a $\psi$-Whittaker functional $\Lambda_\chi$ on the space of $i_P^G(\chi\rho \otimes \sigma)$ by setting
\begin{equation}
\label{eq_heredity}
\Lambda_\chi(f) = \int_N \lambda\left(f(wn)\right)\psi(n)^{-1} dn,
\end{equation}
where $w$ is a representative of the non-trivial element of the Weyl group; in our case, we take $w$ to be the block anti-diagonal matrix
\[
\begin{pmatrix}
\mbox{ } & \mbox{ } & I_k\\
\mbox{ } & I_{2(N-K)+1} & \mbox{}\\
I_k & \mbox{ } & \mbox{}
\end{pmatrix}.
\]
Since $\pi$ and $\pi^-$ are generic, it suffices to determine the action of $T$ on their respective Whittaker functionals if we want to determine how $T$ acts on the $\H$-modules $\Hom_G(\Gamma_\mathfrak{s}, \pi)$ and $\Hom_G(\Gamma_\mathfrak{s}, \pi^-)$.

For any unramified character $\chi$, we have the specialization map $\text{sp}_\chi : \Gamma_\mathfrak{s} \mapsto i_P^G(\chi\rho\otimes \sigma)$ (cf.~\cite[\S 3.1]{heiermann2011operateurs}). The unique (up to scalar multiple) element of $\Hom_G(\Gamma_\mathfrak{s}, \pi)$ factors through $\text{sp}_1 : \Gamma_\mathfrak{s} \to i_P^G(\rho \otimes \sigma)$; similarly, any element of $\Hom_G(\Gamma_\mathfrak{s}, \pi^-)$ factors through $\text{sp}_{\chi_0}$ (recall that $\rho^- = \chi_0\otimes \rho$). Notice that $\Lambda_1$ and $\Lambda_{\chi_0}$ are the Whittaker models of $\pi$ and $\pi^-$, respectively.

To determine the action of $T$ on $\Lambda_\chi$ (for any $\chi$), we must compare $\Lambda_\chi\text{sp}_\chi$ and $\Lambda_\chi \circ \text{sp}_\chi \circ T$. The operator $T$ is defined by the following property:
\[
\text{sp}_\chi T = \varphi \circ J(\chi^{-1}) \circ \text{sp}_{\chi^{-1}}
\]
(cf.~\cite[\S 3.1,3.2]{heiermann2011operateurs}). Here $J(\chi^{-1})$ denotes the standard intertwining operator $i_P^G(\chi^{-1}\rho \otimes \sigma) \to i_P^G(\chi\rho^\vee \otimes \sigma)$. To explain $\varphi$, recall that $\rho$ is assumed to be self-dual. Therefore, we can fix an isomorphism $\varphi: \rho^\vee \mapsto \rho$ and induce to an isomorphism $i_P^G(\chi\rho^\vee \otimes \sigma) \to i_P^G(\chi\rho \otimes \sigma)$ for any unramified $\chi$, which we again denote by $\varphi$ by abuse of notation.

Let $\Lambda_\chi^\vee$ denote the Whittaker functional on $i_P^G(\chi\rho^\vee \otimes \sigma)$ obtained using \eqref{eq_heredity} from a fixed Whittaker functional $\lambda^\vee$ for $\rho^\vee$. By the uniqueness of Whittaker functionals, $\Lambda_\chi \circ \varphi = c\cdot \Lambda_\chi^\vee$ for some constant $c$. Furthermore, since $\varphi$ is induced from an isomorphism $\varphi: \rho^\vee \mapsto \rho$, it follows immediately that $c$ does not depend on $\chi$. Therefore, we have
\[
\Lambda_\chi  \circ \text{sp}_\chi \circ T = c\cdot \Lambda_\chi^\vee \circ J(\chi^{-1}) \circ \text{sp}_{\chi^{-1}}.
\]
Note that there is a natural way to normalize $\varphi$ in such a way that $c = 1$. We denote by $g^\tau$ the transpose of an element $g \in GL_k(F)$ with respect to the anti-diagonal (and with $g^{-\tau}$ its inverse). One can then define a new representation $\rho_1$ by $\rho_1(g) = \rho(g^{-\tau})$. This representation is isomorphic to the contragredient of $\rho$; the advantage is that it acts on $V_\rho$, the space of $\rho$. Furthermore, for any diagonal matrix (i.e.~an element of the maximal torus) $t \in GL_k(F)$, we may conjugate $\rho_1$ to get $\rho_2(g) = {}^t\rho_1(g) =  \rho_1(t^{-1}gt)$. Then $\rho_2 \cong \rho_1$, and with a suitable choice of $t$, $\rho_2$ becomes $\psi$-generic with the same Whittaker functional $\lambda$. For example, assume $\psi$ is given by
\[
\psi(u) = \psi_0(u_{1,2}+\dotsb+u_{k-1,k})
\]
where $\psi_0$ is a non-trivial additive character of $F$, and $u$ is an upper-triangular unipotent matrix with entries $u_{1,2},\dotsc,u_{k-1,k}$ above the main diagonal. Then one checks immediately that $t=\text{diag}(1,-1,\dotsc,(-1)^{k-1})$ gives
\[
\lambda(\rho_2(u)v) = \psi(u)\lambda(v)
\]
for any $v \in V_\rho$. In short, we may assume $
\Lambda_\chi  \circ \text{sp}_\chi \circ T = \Lambda_\chi^\vee \circ J(\chi^{-1}) \circ \text{sp}_{\chi^{-1}}$.

This leads to the second choice we have to make in the construction of $T$: that of the normalization of the intertwining operator $J$. Here we choose the standard normalization introduced by Shahidi; cf.~Theorem 3.1, \cite{shahidi1981certain}. Under this assumption, we have
\[
\Lambda_\chi^\vee \circ J(\chi^{-1}) = \Lambda_{\chi^{-1}}
\]
for every unramified character $\chi$. Thus
\[
\Lambda_\chi  \circ \text{sp}_\chi \circ T = \Lambda_{\chi^{-1}} \circ \text{sp}_{\chi^{-1}}.
\]
With this, we are ready to compare the action of $T$ on $\pi$ and $\pi^-$. 
For $\pi$ we specialize at $\chi = 1$; this gives us
\[
\Lambda_1 \circ \text{sp}_1 \circ T =\Lambda_1 \circ \text{sp}_1,
\]
i.e.~$T$ acts trivially.

For $\pi^-$ we specialize at $\chi_0$. We notice that $\chi_0^{-1} = \chi_0\eta$ for some character $\eta$ such that $\eta \circ \rho \cong \rho$. This shows that $\text{sp}_{\chi^{-1}} = \phi_\eta\circ \text{sp}_{\chi_0}$, where $\phi_\eta$ is the isomorphism $\rho \mapsto \eta \otimes \rho$ defined in \cite[\S 1.17]{heiermann2011operateurs} (again, we induce to $\phi_\eta : i_P^G(\rho \otimes \sigma) \to i_P^G(\eta\rho \otimes \sigma)$ and abuse the notation). Finally, using the uniqueness of Whittaker functionals again, we see that $\Lambda_{\chi\eta} \circ \phi_\eta = d\cdot \Lambda_{\chi}$ for some constant $d$ which does not depend on $\chi$. We can normalize $\phi_\eta$ so that $d=1$; then we have
\[
\Lambda_{\chi_0}  \circ \text{sp}_{\chi_0} \circ T = \Lambda_{\chi_0^{-1}} \circ \text{sp}_{\chi_0^{-1}} = \Lambda_{\chi_0\eta} \circ \phi_\eta\circ \text{sp}_{\chi_0}= \Lambda_{\chi_0} \circ \text{sp}_{\chi_0}.
\]
Therefore, $T$ acts trivially on $\pi^-$ as well.


To summarize, if we use Shahidi's normalization of the standard intertwining operator, and normalize $\varphi$ as we did above, it follows that $T$ acts trivially on both $\pi$ and $\pi^-$. This implies that the Gelfand--Graev module is isomorphic to 
\[
\H \otimes_{\H_0} \epsilon_0
\]
(see \eqref{eq_case_ii_structures}), where $\epsilon_0(T_n) = 1$. Note that this is analogous to our results in Case (iii), because $T_n$ again acts by $q^r$, only this time $r= 0$.

\bigskip

\noindent This completes our analysis of the structure of $\H$. We conclude the section by providing an alternative proof for the following result of \cite{bushnell2003generalized}:
\begin{cor}
We have
\[
\End_\H(\Pi) \cong Z(\H),
\]
the center of $\H$.
\end{cor}
\begin{proof}
Obviously, $Z(H)$ is contained in $\End_\H(\Pi)$, so we need to prove that any element of $\End_\H(\Pi)$ is given by a multiplication with an element $f \in Z(\H)$. We prove the corollary in case (iii); the proof in cases (i) and (ii) is analogous.

We start by recalling that $Z(\H) = \A^W$, the Weyl group invariants of $\A$. Now let $f \in \End_\H(\Pi)$. We have $\End_\H(\Pi) \subseteq \End_\A(\Pi)$, but we know that $\Pi = \A$ as an $\A$-module. Therefore, $f \in  \End_\A(\A) = \A$. Thus, it remains to prove that $f$ is invariant under the action of the Weyl group.

It suffices to prove that $f$ is invariant under the set of simple reflections which generate the Weyl group. In other words, we need to prove that
\[
f^\vee = f \quad \text{and}\quad f^{s_i} = f, \quad i= 1,\dotsc,n-1,
\]
using the notation of \S \ref{subs_structure}. This follows immediately from what we now know about the structure of $\Pi$ as an $\H$-module: $\Pi = \H \otimes_{\H_0} \epsilon$. In other words, we have shown that there exists an element $g\in \A \cong \Pi$ (constructed in \S \ref{subs_H_on_A}) on which the elements $T_1,\dotsc, T_{n-1}$ and $T_n$ act by scalar multiplication with $q^t$, and $q^r$, respectively.

We now look at the commutation relation
\[
T_nf -  f^\vee T_n  = \left((q^r-1) + \frac{1}{X_n}(\sqrt{q}^{r+s}-\sqrt{q}^{r-s})\right)\frac{f - f^\vee}{1-1/X_n^2}
\]
satisfied by $T_n$ and $f$. Applying this to $g$ (recall that $T_ng = q^rg$), and using the fact that $f$ is in $\Hom_\H(\Pi)$ (so that $T_nfg = fT_ng$), we get
\[
(f-f^\vee)q^r\cdot g =  \left((q^r-1) + \frac{1}{X_n}(\sqrt{q}^{r+s}-\sqrt{q}^{r-s})\right)\frac{f - f^\vee}{1-1/X_n^2}\cdot g.
\]
This is an equality in $\A$. Since $q^r \neq \left((q^r-1) + \frac{1}{X_n}(\sqrt{q}^{r+s}-\sqrt{q}^{r-s})\right)\frac{1}{1-1/X_n^2}$ and $g\neq 0$, it follows that $f-f^\vee$ must be $0$. Therefore $f = f^\vee$. We get $f = f^{s_i}$ in the same way, using the commutation relations satisfied by the operators $T_i$. This proves the corollary.
\end{proof}

\begin{appendices}
\section{An equivalence of categories}
\label{sec_appA}

The goal of this section is to prove Theorem \ref{thm_A4} (see below). A partial proof of the result in question is given in Lemma 22 of \cite{bernsteindraft}; however, to the authors' knowledge, no complete proof exists in the literature, so we include it here.

Let $G$ be a reductive $p$-adic group, and $\text{Rep}(G)$ the category of smooth representations of $G$.  Let $P$ be a finitely generated object.  
Let $B=\mathrm{End}_G(P)$. Thus $P$ is a left $G$-module and a
left $B$-module.  For every smooth $G$-module $V$, 
\[ 
\mathfrak F(V)= \Hom_G(P, V) 
\] 
is naturally a right $B$-module, and for every right $B$-module $N$,  
\[ 
\mathfrak G(N) = N\otimes_B P 
\] 
is naturally a smooth $G$-module. Thus we have two functors.  Let $V$ be a smooth $G$-module, and $N$ a right $B$-module.  
Observe that $\Hom_G(V, P)$ is a left $B$-module. We have a map 
\[ 
N\otimes_B \Hom_G(V, P)  \rightarrow \Hom_G( V, N\otimes_B P) 
\] 
defined by $n\otimes t \mapsto T$ where $T(v)= n\otimes t(v)$ for all $v\in V$.  We have, see 2.7 in \cite{DI}: 

\begin{lem} If $V$ is a finitely generated and projective  $G$-module then 
\[ 
N\otimes_B \Hom_G(V, P)  \rightarrow \Hom_G( V, N\otimes_B P) 
\] 
is an isomorphism.  
\end{lem} 

\begin{proof} Since $V$ is finitely generated, there exists an open compact subgroup $K$ and an integer $r$ such that 
$V$ is a quotient of $\cind_K^G (1)^r$, in fact, a summand since it is projective. Thus it suffices to prove the
lemma for $V=\cind_K^G (1)$. In this case both spaces are equal to $N\otimes_B P^K$. 
\end{proof} 

\begin{cor} Assume $P$ is finitely generated and projective. For every $B$-module $N$, the map 
\[ 
N \rightarrow \Hom_G(P, N\otimes_B P) = (\mathfrak F\circ \mathfrak G )(N) 
\] 
defined by $n\mapsto (p\mapsto n\otimes p)$ is an isomorphism. 
\end{cor} 
\begin{proof} Use $V=P$ in the lemma. 
\end{proof} 

\begin{defn} Let $\mathcal C$ be a categorical direct summand of $\text{Rep}(G)$. A finitely generated and  projective object $P$ in 
$\mathcal C$ is called a pro-generator if, for every object $V$ in $\mathcal C$, the natural map 
\[ 
\Hom_G(P, V) \otimes_{\mathbb C} P \rightarrow V 
\] 
is surjective.  
\end{defn} 

\begin{lem} Let $P$ is a pro-generator of $\mathcal C$. Let $B=\mathrm{End}_G(P)$.  For every object $V$ in $\mathcal C$ the map 
\[ 
(\mathfrak G\circ \mathfrak F )(V) = \Hom_G(P, V)\otimes_B P  \rightarrow V
\] 
defined by $t\otimes p\mapsto t(p)$ is an isomorphism.  
\end{lem} 
\begin{proof} 
This is trivial if $V=P$. Now assume that  $V$ is a direct sum of (infinitely many) copies of $P$.  Since $P$ is finitely generated, the lemma is also true for such $V$.  
In general, since $P$ is a pro-generator, we have an exact sequence 
\[ 
Y\otimes_{\mathbb C} P \rightarrow X\otimes_{\mathbb C} P \rightarrow V \rightarrow 0 
\] 
 where $X$ and $Y$ are vector spaces. Since $P$ is projective, 
 \[ 
\Hom_G(P, Y\otimes_{\mathbb C} P)  \rightarrow \Hom_G(P,X\otimes_{\mathbb C} P) \rightarrow \Hom_G(P,V )\rightarrow 0 
\]  
is exact. Since tensoring is right exact,  
\[ 
\Hom_G(P, Y\otimes_{\mathbb C} P) \otimes_B P  \rightarrow \Hom_G(P,X\otimes_{\mathbb C} P) \otimes_B P\rightarrow \Hom_G(P,V )\otimes_B P \rightarrow 0 
\] 
is exact. Since the first and the second term are respectively isomorphic to $Y\otimes_{\mathbb C} P$ and $X\otimes_{\mathbb C} P$, 
the lemma follows by a simple diagram chase. 
\end{proof} 

Combining the above, we have: 

\begin{thm}
\label{thm_A4}Let $\mathcal C$ be a categorical direct summand of $\textnormal{Rep}(G)$. Let $P$ be a pro-generator of $\mathcal C$. 
Then the functor 
\[ 
\mathfrak F(V)= \Hom_G(P, V),  
\] 
where $V$ is an object in $\mathcal C$, is an equivalence of $\mathcal C$ and the category of right ${\rm End}_G(P)$-modules. 
\end{thm}

\section{An isomorphism of projective generators}
\label{sec_appB}

Let\footnote{In this section, we freely use the notation introduced in the main body of the paper, particularly \S\ref{subs_Hecke}} $\mathfrak{s}$ be the inertial class of cuspidal data $(M,\sigma)$. To any such  $\mathfrak{s}$ Bushnell and Kutzko \cite{bushnell1998smooth} attach a type $(J,\lambda)$, where $J$ is a compact subgroup of $G$, and $\lambda$ is an irreducible representation of $J$. Then $\cind_J^G \lambda$ is a projective generator for $\text{Rep}_{\mathfrak{s}}(G)$. We are interested in the structure of the Hecke algebra $\mathcal{H}(G,\lambda)= \End_G(\cind_J^G \lambda)$. In some special cases, this Hecke algebra has already been described in \cite{miyauchi2014semisimple}.

In this section, we show that---for a suitable choice of type $(J,\lambda)$---the Hecke algebra $\mathcal{H}(G,\lambda)=\End_G(\cind_J^G \lambda)$ is isomorphic to the algebra $\H_\mathfrak{s} = \End_G(\Gamma_\mathfrak{s})$ constructed in Section \ref{subs_Hecke}. This follows from

\begin{thm}
\label{thm_B}
Assume that the residue characteristic of F is different from $2$. Let $\mathfrak{s}=[(M,\sigma)]$ be an inertial equivalence class in $G$. There exists an $\mathfrak{s}$-type $(J,\lambda)$ such that the generators $\Gamma_\mathfrak{s}$ and $\cind_J^G\lambda$ are isomorphic.
\end{thm}

\begin{proof}
We use the theory of covers developed by Bushnell and Kutzko. Any inertial equivalence class $\mathfrak{s} = [(M,\sigma)]$ in $G$ also determines a (cuspidal) inertial equivalence class $\mathfrak{s}_M = [(M,\sigma)]$ in $M$. Let $(J,\lambda)$ be a type for $\mathfrak{s}$ and $(J_M,\lambda_M)$ a type for $\mathfrak{s}_M$. We say that the $(J,\lambda)$ is a cover of the type $(J_M,\lambda_M)$ if $J$ decomposes with respect to $M$ (in particular, $J_M = J \cap M$ and $\lambda_M = \lambda|_M$) and the equivalence of categories $\text{Rep}_{\mathfrak{s}}(G) \to \mathcal{H}(G,\lambda)$-Mod commutes with parabolic induction and the Jacquet functor in the appropriate sense (see Definition 8.1 and paragraph 5 of Introduction in \cite{bushnell1998smooth}). We then have the following.

\begin{lem}[Theorem 7.9 (iii) of \cite{bushnell1998smooth}]
\label{prop1}
Let $P$ be any parabolic subgroup with Levi factor $M$. For any smooth representation $V \in \textnormal{Rep}(G)$, the Jacquet functor with respect to $P$ induces an isomorphism
\[
V^\lambda = (V_N)^{\lambda_M}. 
\]
Here $V^\lambda$ denotes the $\lambda$-isotype of $V$, i.e.~the sum of all $G$-invariant subspaces of $V$ isomorphic to $\lambda$.
\end{lem}

We use this to reduce the proof of Theorem \ref{thm_B} to the case of cuspidal components.

\begin{lem}
Let $(J,\lambda)$ be a type for $\mathfrak{s} = [(M,\pi)]$ in $G$, and let $(J_M,\lambda_M)$ be a type for $\mathfrak{s}_M =  [(M,\pi)]$ in $M$. Assume that $(J,\lambda)$ is a cover of $(J_M,\lambda_M)$.

If the Bernstein generator $\Gamma_{\mathfrak{s}_M}$ is isomorphic to the Bushnell--Kutzko generator $\cind_{J_M}^M \lambda_M$ for the cuspidal component $\textnormal{Rep}_{\mathfrak{s}_M}(M)$, then we also have an isomorphism of generators for the component $\textnormal{Rep}_{\mathfrak{s}}(G)$.
\end{lem}

\begin{proof}
Lemma \ref{prop1} shows that we have
\[
\Res_{J_M}^J((\Res_{J}^G V )^\lambda) = (\Res_{J_M}^M r_N(V))^{\lambda_M}
\]
for any $G$-module $V$.
Here $\Res_H^G$ denotes the restriction functor from $G$ to $H$, and $r_N$ the Jacquet functor with respect to $P = MN$. In other words, we get the following isomorphism of functors $\text{Rep}(G) \to \text{Rep}(J_M)$:
\[
\label{eq_equi1}\tag{\textasteriskcentered}
\Res_{J_M}^J \circ (\lambda\text{-iso}) \circ \Res_{J}^G = (\lambda_M\text{-iso}) \circ \Res_{J_M}^M \circ r_N,
\]
where we have used $\lambda\text{-iso}$ (resp.~$\lambda_M\text{-iso}$) to denote taking the $\lambda$- (resp.~$\lambda_M$-) isotype.

All of the above functors have left adjoints:
\begin{itemize}[label=\textemdash]
\item $\cind_{J_M}^J$ and $\cind_{J_M}^M$ for $\Res_{J_M}^J$ and $\Res_{J_M}^M$, respectively;
\item $i_{\overline{P}}^G$ for $r_N$ (this is the Bernstein form of Frobenius reciprocity; here $\overline{P} = M\overline{N}$ is the parabolic subgroup opposite to $P$)
\item $\lambda$-iso and $\lambda_M$-iso are self-adjoint, because we are working with (necessarily semisimple) representations of compact groups $J$ and $J_M$.
\end{itemize}
Since adjoints are unique (up to equivalence), taking the adjoint of \eqref{eq_equi1} we get
\[
\cind_J^G\circ (\lambda\text{-iso}) \circ \cind_{J_M}^J = i_{\overline{P}}^G \circ \cind_{J_M}^M \circ (\lambda_M\text{-Iso}).
\]
We now apply both sides of the above equality to $\lambda_M$. On the right-hand side, we get $i_{\overline{P}}^G(\cind_{J_M}^M \lambda_M)$. By the assumptions from the statement of the lemma, we have $\cind_{J_M}^M \lambda_M =\Gamma_{s_M}$; therefore, $i_{\overline{P}}^G(\cind_{J_M}^M \lambda_M)$ is exactly the Bernstein generator $i_{\overline{P}}^G(\Gamma_{s_M}) = \Gamma_s$. Here we used the fact that the construction of the Bernstein generator does not depend on the choice of parabolic $P$ (we choose $\overline{P}$) with fixed Levi $M$ (cf.~\cite[Proposition 35]{bernsteindraft}).

On the left-hand side, we get $\cind_{J}^G((\cind_{J_M}^J \lambda_M)^{\lambda})$. However, Frobenius reciprocity gives us $\dim \Hom_J(\cind_{J_M}^J \lambda_M, \lambda) = \dim \Hom_{J_M}(\lambda_M, \lambda|_M) = 1$, which follows from $\lambda|_M = \lambda_M$. Therefore $(\cind_{J_M}^J \lambda_M)^{\lambda} = \lambda$, and the left-hand side becomes $\cind_{J}^G(\lambda)$, i.e.~the Bushnell--Kutzko generator. Thus
\[
\cind_{J}^G(\lambda) \cong \Gamma_\mathfrak{s},
\]
as claimed.
\end{proof}

The above lemma allows us to focus on cuspidal components of the form $\mathfrak{s}_M = [(M,\sigma)]$ in $M$. If we want to prove the isomorphism of generators in general, it remains to prove that the generators of the cuspidal components are isomorphic. In other words, we would like to show that 
\[
\cind_{M^\circ}^M \sigma_0 = \cind_{J_M}^M \lambda_M,
\]
where $\sigma_0$ is an (any) irreducible constituent of $\sigma|_{M^\circ}$. Notice that this is equivalent to $\cind_{J_M}^{M^\circ} \lambda_M$ being irreducible. Indeed, since $\lambda_M$ is a type, $\cind_{J_M}^M \lambda_M$ possesses $\sigma$ as a quotient; therefore (by Frobenius reciprocity), $\cind_{J_M}^{M^\circ} \lambda_M$ contains at least one irreducible summand of $\sigma|_{M^\circ}$. If it contained anything else, then $\cind_{J_M}^M \lambda_M$ would be strictly larger than $\cind_{M^\circ}^M \sigma_0$.

We know that $\sigma|_{M^\circ}$ is semisimple. It is not hard to see that, since $\lambda_M$ is a type, any irreducible subquotient of $\cind_{J_M}^{M^\circ} \lambda_M$ must appear in $\sigma|_{M^\circ}$. Therefore, all its irreducible subquotients are cuspidal, and it admits a central character. It follows that $\cind_{J_M}^{M^\circ} \lambda_M$ is semisimple as well. Thus, in order to prove the irreducibility of $\cind_{J_M}^{M^\circ} \lambda_M$, it suffices to show that
$
\dim \Hom_{M^\circ}(\cind_{J_M}^{M^\circ} \lambda_M,\sigma|_{M^\circ}) = 1
$, i.e. (using Frobenius reciprocity)
\[
\dim \Hom_{M}(\cind_{J_M}^{M} \lambda_M,\sigma) = 1.
\]
It turns out that we can achieve this when $G$ is the general linear group or a classical group over over $F$, under the assumption that $F$ has odd residue characteristic. By the results of Bushnell--Kutzko \cite{bushnell1998smooth} (for $\GL_n(F)$) and Miyauchi--Stevens \cite{miyauchi2014semisimple} (for classical groups), any component $\mathfrak{s}$ in $G$ admits a type $(J,\lambda)$ which is moreover a cover of a \emph{cuspidal} type $(J_M,\lambda_M)$ for $\mathfrak{s}_M$, in the sense of \cite[Definition 4.3]{miyauchi2014semisimple}. In that setting, we have the desired result (we drop the subscript $M$ to simplify notation):

\begin{lem}
Let $(J,\lambda)$ be a cuspidal type for a cuspidal component $\mathfrak{s} = [(M,\sigma)]$ in $M$. Then 
\[
\dim \Hom_{M}(\cind_{J}^{M} \lambda,\sigma) = 1.
\]
Consequently, the generators $\Gamma_\mathfrak{s}$ and $\cind_{J}^M\lambda$ are isomorphic in this case.
\end{lem}

\begin{proof}
The fact that $\lambda$ is a cuspidal type means that we can extend it to an irreducible representation $\Lambda$ of $N(J)$ (the normalizer of $J$ in $M$), with $\sigma = \cind_{N(J)}^M \Lambda$. Using Frobenius reciprocity and Mackey theory (provided by \cite[\S 5.5]{vigneras1996representations} in this setting), we get 
\begin{align*}
 \Hom_{M}(\cind_{J_M}^{M} \lambda,\pi) &\cong  \Hom_{M}(\cind_{J}^{M} \lambda,\cind_{N(J)}^M \Lambda)\\
 &\cong \Hom_{J}(\lambda,\bigoplus_x \cind_{J\cap {N(J)^x}}^J \text{Res}_{J\cap {N(J)^x}}^{N(J)^x} \Lambda^x)
\end{align*}
where the sum is taken over a set of double coset representatives in $J\backslash M /N(J)$. Fixing one such $x$, we see that
\[
\Hom_{J}(\lambda, \cind_{J\cap {N(J)^x}}^J \text{Res}_{J\cap {N(J)^x}}^{N(J)^x} \Lambda^x) \cong \Hom_{J\cap {N(J)^x}}(\lambda, \text{Res}_{J\cap {N(J)^x}}^{N(J)^x} \Lambda^x)
\]
(here we are using Frobenius reciprocity for a compact group, so that restriction is also a left adjoint for $\cind$). Conjugating everything on the right-hand side, we see that the above $\Hom$-space is isomorphic to
\[
\Hom_{J^x\cap N(J)}(\lambda^x, \text{Res}_{J^x\cap {N(J)}}^{N(J)} \Lambda) = \Hom_{J^x\cap J}(\lambda^x, \lambda).
\]
In the above equality we use the fact that $J^x \cap N(J) = J^x\cap J$: the group $J^x \cap N(J)$ is compact, so it is contained in the maximal compact subgroup of $N(J)$, i.e.~$J$.

Now since $\lambda$ is a cuspidal type, its intertwining is equal to the normalizer of $J$ in $G$, $N(J)$---this is Proposition 6.18 in \cite{stevens2008supercuspidal}. Therefore, $\Hom_{J^x\cap J}(\lambda^x, \lambda)$ is non-zero if and only if $x \in N(J)$, and in that case it is equal to $\Hom_{J}(\lambda, \lambda)$. In short, only the trivial coset contributes to the above sum of $\Hom$-spaces, and we have
\[
 \Hom_{M}(\cind_{J_M}^{M} \lambda,\pi) \cong \Hom_{J}(\lambda, \lambda) = \mathbb{C},
\]
which we needed to prove.
\end{proof}

\noindent This completes the proof of Theorem \ref{thm_B}.
\end{proof}

\end{appendices}

\bibliographystyle{hsiam}
\bibliography{bibliography}

\bigskip
\end{document}